\documentclass[reqno]{amsart}

\usepackage{amsmath}
\usepackage{amssymb}
\usepackage{graphicx}
\usepackage{color}
\usepackage[all]{xy}
\usepackage{amsthm}
\usepackage{tikz}
\usetikzlibrary{shapes.geometric}
\usepackage{caption}
\usepackage{subcaption}
\usepackage[colorlinks,linkcolor=blue,citecolor=blue,pdfstartview=FitH]{hyperref}

\usepackage{cleveref}

\theoremstyle{plain}
\newtheorem{thm}{Theorem}[section]
\newtheorem{lem}[thm]{Lemma}
\newtheorem{prop}[thm]{Proposition}

\newtheorem{question}[thm]{Question}
\newtheorem{cor}[thm]{Corollary}
\newtheorem{rem}[thm]{Remark}

\newtheorem{conj}[thm]{Conjecture}
\newtheorem*{claim}{Claim}

\newcommand{\co}{\colon\thinspace}

\newcommand{\C}{\mathbb{C}}

\newcommand{\Q}{\mathbb{Q}}
\newcommand{\R}{\mathbb{R}}
\newcommand{\Z}{\mathbb{Z}}
\newcommand{\OO}{\mathbb{O}}

\newcommand{\PSLTC}{PSL(2,\C)}
\newcommand{\PSLTOTh}{PSL(2,\OO_3)}
\newcommand{\PSLTOone}{PSL(2,\OO_1)}
\newcommand{\PGLTOone}{PGL(2,\OO_1)}
\newcommand{\PSLTOD}{PSL(2,\OO_d)}

\newcommand{\Sth}{\mathbb{S}^3}
\newcommand{\Hth}{\mathbb{H}^3}

\newcommand{\orb}{\mathcal{O}}
\newcommand{\orbQ}{\mathcal{Q}}

\definecolor{bettergreen}{rgb}{0,0.6,0.4}
\definecolor{purple}{rgb}{0.4,0,0.6}

\begin{document}

\title{Cusp types of quotients of hyperbolic knot complements}

\author{ Neil R. Hoffman}
 \address{Department of Mathematics, Oklahoma State University, Stillwater, OK}
\email[]{neil.r.hoffman@okstate.edu}


\maketitle

\begin{abstract}
This paper completes a classification of the types of orientable and non-orientable cusps that can arise in the quotients of hyperbolic knot complements. In particular, $S^2(2,4,4)$ cannot be the cusp cross-section of any orbifold quotient of a hyperbolic knot complement. Furthermore, if a knot complement covers an orbifold with a $S^2(2,3,6)$ cusp, it also covers an orbifold with a $S^2(3,3,3)$ cusp. We end with a discussion that shows all cusp types arise in the quotients of link complements. 
\end{abstract}


\section{Introduction}

There are five orientable Euclidean 2-orbifolds:
 $$T^2, S^2(2,2,2,2), S^2(2,3,6), S^2(3,3,3) \mbox{ and } S^2(2,4,4).$$ 
 The figure 8 knot complement covers orbifolds with four of the five types of cusps (all but $S^2(2,4,4))$. In keeping with the standard terminology, we say that $S^2(2,3,6)$, $S^2(3,3,3)$ and $S^2(2,4,4)$ are the cross-sections of \emph{rigid cusps} because they have a unique geometric structure up to Euclidean similarity. We also say $T^2$ and $S^2(2,2,2,2)$ correspond to \emph{non-rigid cusps}, which up to similarity have a two dimensional (real) parameter space of possible cusp shapes.  
The aforementioned quotients of the figure 8 knot complement to orbifolds with $S^2(2,3,6)$ and $S^2(3,3,3)$ cusps can be easily constructed by analyzing the symmetries of the underlying space of the complement, two regular ideal tetrahedra (see \cite{NR92b} for example). Similarly, there are two knot complements that decompose into regular ideal dodecahedra \cite{AiRu} and each of these knot complements also admits a quotient to an orbifold with a $S^2(2,3,6)$ cusp and to an orbifold with a $S^2(3,3,3)$ cusp (see \cite[\S 9]{NeumReid}, and \cite{Hoffman} for more background).   Collectively, these three examples are the only  known examples of hyperbolic knot complements that cover orbifolds with rigid cusps. Again, for each of the three examples, we can find a quotient with a with  $T^2, S^2(2,2,2,2), S^2(2,3,6)$ or $ S^2(3,3,3) $ cusp.  Curiously missing from this list is a hyperbolic knot complement which covers an orbifold with a $S^2(2,4,4)$ cusp. The main theorem shows that such a cover cannot occur.

\begin{thm}\label{thm:no_s244cusps}
Let $\orbQ$ be an orbifold covered by a hyperbolic knot complement, then $\orbQ$ does not have a $S^2(2,4,4)$ cusp. 
\end{thm}

This theorem is also relevant to the larger question of which knot complements admit hidden symmetries (see \Cref{sect:background}). It is conjectured by Neumann and Reid that only the figure 8 knot complement and the dodecahedral knot complements exhibit this property (see  \cite[Conjecture 1.1]{BoiBoCWa2} for example). In that context, we can appeal to a similar argument to the proof of the main theorem and show that if a knot complement covers an orbifold with a $S^2(2,3,6)$, then it also covers an orbifold with a $S^2(3,3,3)$ cusp. Of course, we see this phenomenon in the examples discussed above. Thus, generically, if a hyperbolic knot complement covers an orbifold with a $S^2(2,3,6)$ cusp, we can build a diagram of orbifolds covers where an orbifold with a  $S^2(3,3,3)$ cusp is an intermediate cover (see in Figure \ref{fig:Figure8Quotients}). In the known examples, for the figure 8 knot complement $n_0=1$, $n_1=2$, $6n_2=6$. In the case of the dodecahedral knots, $n_0=1$, but in general $n_1$ and $n_2$ are determined by the relevant symmetry groups (see \cite{Hoffman}). More importantly, an orbifold with an $S^2(3,3,3)$ also arises, as evidenced below.

\begin{thm}\label{thm:s236case}
 Let $f \co M \rightarrow \orbQ$, where $M$ is a manifold covered by a hyperbolic knot complement.
If $\orbQ$ has a $S^2(2,3,6)$ cusp, then $(\pi_1^{orb}(\orbQ))^{ab} \cong \Z/2\Z$, $deg(f) =24n$, $n \geq 1$, and $M$ covers the double cover of $\orbQ$, which has
a $S^2(3,3,3)$ cusp.
\end{thm}

\begin{figure}
    \centerline{
\xymatrix{
\Sth \setminus K \ar[d]_{n_0} & &\\
M \ar[drr]^{n_1} \ar[ddrr]^{12n_3} \ar[dddrr]_{12n_3,24n_3}   &\\
& & \orbQ_T \ar[d]^{6n_2}\\
& & \orbQ_3 \ar[d]^{1,2}\\
& & \orbQ\\
}}
\caption{\label{fig:Figure8Quotients} The structure of quotients of a manifold covered by a knot complement that admits hidden symmetries in light of \Cref{thm:s236case} and \Cref{thm:no_s244cusps}. Here $\orbQ_3$ has an $S^2(3,3,3)$ cusp and $\orbQ_T$ has a torus cusp. We point out that $n_i \geq 1$ and the degree of the cover of $\orbQ_3$ to $\orbQ$ could be $1$ or $2$.  }
\end{figure}

\Cref{thm:s236case} is sharp in the sense that there is a $24$-fold cover by the figure 8 knot complement to the minimum volume orientable orbifold in its commensurability class $\Hth/PGL(2,\mathbb{O}_3)$.  However, for each of the dodecahedral knot complements the analogous cover is degree $120$.

The main results of this paper provide evidence for the Rigid Cusp Conjecture of Boileau, Boyer, Cebanu, and Walsh \cite[Conjecture 1.3]{BoiBoCWa2}.

\begin{conj}[Rigid Cusp Conjecture] If a hyperbolic knot complement $\Sth \setminus K$ covers an orbifold with a rigid cusp, $\Sth \setminus K$ covers an orbifold with a $S^2(2,3,6)$ cusp.
\end{conj}

 \Cref{thm:no_s244cusps} reduces the conjecture to answering the following question, which would imply a partial converse of \Cref{thm:s236case}. 
 
\begin{question}
If $\Sth \setminus K$ covers an orbifold with a $S^2(3,3,3)$ cusp, must it also cover an orbifold with a $S^2(2,3,6)$ cusp?
\end{question}

%

We end this section with a theorem that effectively serves as a summary of the arguments of this paper so far.

\begin{thm}\label{thm:possible_non_or_cusps}
Let $\Sth \setminus K$ cover a orbifold $\orbQ$. Then the cusp of $\orbQ$ is either:
orientable and one of $T^2$, $S^2(2,2,2,2)$,  $S^2(2,3,6)$, or $S^2(3,3,3)$,
or non-orientable and one of 
$K^2$, $RP^2(2,2)$, $D^2(;2,3,6)$, $D^2(;3,3,3)$ or $D^2(3;3)$.
\end{thm}

We remark that the discussions above and in \Cref{sect:non_or} show that this is precisely the set of cusp types that can be covered by hyperbolic knot complements. However, for all rigid cusps there is a link complement which admits a quotient with at least one cusp of that type.

\subsection{Acknowledgements:} 
This work was partially supported by grant from the
Simons Foundation (\#524123 to Neil R. Hoffman). The author also wishes to thank Christian Millichap and William Worden for helpful conversations and Eric Chesebro, Michelle Chu, Jason Deblois, Pryadip Mondal, and Genevieve Walsh for agreeing to hear a video presentation of a talk based on this work after an AMS sectional meeting was cancelled in March 2020. Their feedback (and accountability as an audience) was both considerably helpful and greatly appreciated.  We thank the referee for numerous helpful suggestions including the prompt to add \Cref{sub:link_realization}.

\section{background}\label{sect:background}

We refer the reader to \cite[Chapter 13]{Th_notes} for more information on orbifolds and we will appeal to Thurston's notation conventions for 2-orbifolds established in that reference.  Namely for $2-$orbifolds, we say $F(a_1,...a_m;b_1,...b_n)$ is an orbifold with underlying space $F$, cone points of orders $a_i$ and corner reflectors marked by $b_j$. A sufficiently small neighborhood of a cone point is isometric to $D^2/(\Z/a_i \Z)$ and a sufficiently small neighborhood of a corner reflector is isometric to $D^2/(D_{2b_j})$ where $D_{2b_j}$ is a dihedral group of order $2b_j$. The lone departure from these conventions is explained in \Cref{fig:relevant_cusps} and \Cref{rem:notation}, which concerns simple closed curves with point only fixed a single reflection (i.e. there are no corner points along the curve).

We begin by defining terminology to frame the discussion.  A cusped hyperbolic $3$-orbifold $\orbQ$ is \emph{arithmetic} if and only if $\pi_1^{orb}(\orbQ)$ is commensurable with a subgroup of $\PSLTOD$ for some $d$. We refer the reader to \cite[Chapter 8]{MR03} for further background. 
There is a nice dichotomy of for arithmetic and non-arithmetic orbifolds. In the context of this paper, Reid \cite{ReidFig8} showed that the figure 8 knot complement is the only arithmetic knot complement. More generally, Margulis \cite{Margulis1991} showed that $\orbQ= \Hth/\Gamma$ is non-arithmetic, then $$Comm^+(\Gamma) = \{g \in Isom^+(\Hth) | [\Gamma \co \Gamma \cap g\Gamma g^{-1}]< \infty,  [g\Gamma g^{-1} \co \Gamma \cap g\Gamma g^{-1}]< \infty\}$$
 is discrete. In terms a knot complement $\Sth\setminus K$, either $K$ is the figure 8 knot or all (orientable) orbifolds commensurable with $\Sth \setminus K \cong \Hth/\Gamma_K$ will cover the commensurator quotient $\Hth/Comm^+(\Gamma_K)$.

A hyperbolic manifold $M$ admits \emph{hidden symmetries} if there exist a cover of $M$, $\tilde{M}$ such that $\tilde{M}$ admits a symmetry that is not the lift of a deck transformation of $M$. For a (non-arithmetic) hyperbolic knot complement $\Sth \setminus K$, 
Neumann and Reid showed that
 admitting hidden symmetries is equivalent to covering a rigid cusped orbifold (see \cite[\S 9]{NeumReid}).  

\subsection{The structure of the orbifold}

An orientable, (finite volume) hyperbolic 3-orbifold $\orbQ$ can be described as the quotient $\Hth/\Gamma$ (with $vol(\Hth/\Gamma) < \infty)$.  If $\orbQ$ is non-compact, then $\orbQ$ has cusps of the form $T^2 \times [0,\infty)$, $S^2(2,2,2,2) \times [0,\infty)$, $S^2(2,4,4) \times [0,\infty)$, $S^2(2,3,6) \times [0,\infty)$ or $S^2(3,3,3) \times [0,\infty)$. 

The \emph{singular set} of $\orbQ$ denoted by $\Sigma(\orb)$ is the set of points in the quotient $\orbQ \cong \Hth/\Gamma$ with non-trivial point stablizers in $\Gamma$. The \emph{underlying space} $|\orbQ|$ of $\orbQ$ is the 3-manifold determined by ignoring labels on $\Sigma(\orb)$.   It is well established that $\Sigma(\orbQ)$ is an embedded trivalent graph in $|\orbQ|$ (see for example \cite{boileau2005geometrization,cooper2000three}). 

As we consider knots in $S^3$ and their symmetries, the following definition will be useful later. An \emph{orbilens space} is a orientable orbifold cyclic quotient of $S^3$. For a thorough treatment of orbilens spaces in the context of knot complement quotients we refer the reader to \cite[\S 3]{BBoCWaGT}.

\subsection{Elements of $\pi_1^{orb}(S^2(2,3,6))$}\label{sub:translation_structure}
 We can be very explicit about our cusp groups and their action by isometries on the Euclidean plane. We start out by introducing relevant details of a $S^2(2,3,6)$ cusp group. 

Assume $\pi_1^{orb}(S^2(2,3,6))= \langle a, b | a^6,b^3,(ab)^2\rangle$. We can find a discrete faithful representation of this group into $Isom(E^2)$ using 
$$a \mapsto \begin{pmatrix} \cos(\pi/3) & \sin(\pi/3)\\ -\sin(\pi/3) & \cos(\pi/3) \end{pmatrix} \begin{pmatrix} x\\ y \end{pmatrix}$$ and 
$$b \mapsto \begin{pmatrix} \cos(2\pi/3) & \sin(2\pi/3)\\ -\sin(2\pi/3) & \cos(2\pi/3) \end{pmatrix} \begin{pmatrix} x\\ y \end{pmatrix} + \begin{pmatrix} \frac{1}{2}\\ \frac{1}{2\sqrt{3}} \end{pmatrix}$$

Notice the maximal abelian subgroup (aka the subgroup of translations) in this group is generated by:
$t_1 = ba^{-2} \mapsto  \begin{pmatrix} x+ \frac{1}{2}\\ y+\frac{\sqrt{3}}{2} \end{pmatrix}$, $t_2 = b^{-1}a^2 \mapsto  \begin{pmatrix} x+ 1 \\ y \end{pmatrix}$.

\subsection{Elements of $\pi_1^{orb}(S^2(2,4,4))$}\label{sub:translation_structure244}
The introduction listed the three hyperbolic knot complements known to cover orbifolds with $S^2(2,3,6)$ cusps. Before we exhibit the obstruction to a knot complement covering an orbifold with $S^2(2,4,4)$ cusp, we give background on this cusp type as well.


Analogously to the argument above, we will begin with a discrete faithful representation of the fundamental group of $S^2(2,4,4)$ into $Isom(E^2)$. Of course, a different embedding of this this group into $Stab(\infty) \subset \PSLTC$ is used in order to realize it as the peripheral subgroup subgroup of a hyperbolic 3-orbifold (see \cite{Hoffman_hidden} for example). 

Assume $\pi_1^{orb}(S^2(2,4,4))= \langle c, d | c^4,d^2,(cd)^4 \rangle$. We observe that if we map $$c \mapsto \begin{pmatrix} \cos(\pi/2) & \sin(\pi/2)\\ -\sin(\pi/2) & \cos(\pi/2) \end{pmatrix} \begin{pmatrix} x\\ y \end{pmatrix}$$
and 
$$d \mapsto \begin{pmatrix} \cos(\pi) & \sin(\pi)\\ -\sin(\pi) & \cos(\pi) \end{pmatrix} \begin{pmatrix} x\\ y \end{pmatrix} + \begin{pmatrix} 1\\ 0 \end{pmatrix}$$
we have our desired discrete faithful representation of this group into $Isom(E^2)$.

Notice the maximal abelian subgroup in this group is generated by:
$t_1  = c^2 d  \mapsto  \begin{pmatrix} x+1\\ y \end{pmatrix}$, $t_2 = cdc \mapsto  \begin{pmatrix} x  \\ y +1 \end{pmatrix}.$


%

\subsection{Improvements to degree bounds and a structure theorem}

 We will establish a useful improvement to the covering degree bounds of \cite[Lemma 5.5]{Hoffman_hidden}.
 We point out that this result is a structure theorem in the sense that it implies that knot complements which cover $S^2(2,3,6)$ cusped orbifolds also cover
 $S^2(3,3,3)$ cusped orbifolds. For comparison, the Rigid Cusp Conjecture of Boileau, Boyer, Cebanu, and Walsh \cite[Conjecture 1.3]{BoiBoCWa2} postulates that every hyperbolic knot complement admitting hidden symmetries covers an orbifold with a $S^2(2,3,6)$ cusp.

\begin{proof}[Proof of \Cref{thm:s236case}]
Let $\orbQ$ be a quotient of a hyperbolic knot complement $\Sth\setminus K$. Assume that $\orbQ$ has a $S^2(2,3,6)$ cusp. 
Further assume that $\Sth \setminus K$ covers $M$ and denote by $\Gamma_K = \pi_1(\Sth \setminus K) \subset \pi_1^{orb}(\orbQ) \subset \PSLTC$. We assume that $\Gamma_K$ and $\pi_1^{orb}(\orbQ)$ are identified with a discrete faithful representation of $\pi_1(\Sth\setminus K)$.

Since $\orbQ$ and $\Sth \setminus K$  both have one cusp,  $ \pi_1^{orb}(\orbQ)$ decomposes into cosets index by peripheral elements, and so $ \pi_1^{orb}(\orbQ) = P_6 \cdot \Gamma_K$ where $P_6$ is the peripheral subgroup of $\orbQ$.

We can further assume that $\Gamma_K = \langle \mu_1, ... \mu_n | r_1, ... r_m \rangle$, where $\mu_i$ are meridians of $K$ and each $r_j$ is a relation of $\Gamma_K$ (e.g. coming from a Wirtinger presentation.) Finally, assume $P_6 = \langle a, b | a^6,b^3,(ab)^2 \rangle$ and that $\mu_1 \in P_6$.   
  
Thus, we appeal to the structures of $P_6$ and $\Gamma_K$ to define a left coset action on $\Gamma_K$ as a subgroup of $ \pi_1^{orb}(\orbQ)$. From this coset action, we can deduce properties of presentation for $ \pi_1^{orb}(\orbQ)$. Specifically, we have

$$  \pi_1^{orb}(\orbQ) = \langle a, b, \mu_1, ... \mu_n | a^6,b^3,(ab)^2, r_1, ... r_m,  a\mu_j a^{-1}\gamma_{a,j}^{-1}, b\mu_j b^{-1}\gamma_{b,j}^{-1} \rangle$$
such that each $ \gamma_{a,j}^{-1}$, $\gamma_{b,j}^{-1}$ is a parabolic in $\pi_1^{orb}(\orbQ)$. Here we point out that since $ \pi_1^{orb}(\orbQ) = P_6 \cdot \Gamma_K$, these relations together with $a^6,b^3,(ab)^2 $ and $r_1,..., r_m$ form a complete set of relations for $ \pi_1^{orb}(\orbQ)$.

 Notice since $\Gamma_K$ has one orbit of parabolic fixed points (which is the same set of parabolic fixed points as $ \pi_1^{orb}(\orbQ)$),  $\gamma_{a,j}$ and $\gamma_{b,j}$ can be expressed as $w t_1^r t_2^s w^{-1}$, where $w \in \Gamma_K$ and $t_1 = ba^{-2}$, $t_2 = b^{-1}a^2$ (see \Cref{sub:translation_structure} for more details). 
Thus, in the quotient obtained by modding out by the normal closure of all parabolic elements of $P_6 \cdot \Gamma_K$ and $b$, $\Gamma_K$ becomes trivial, and $P_6 \twoheadrightarrow \Z/2\Z$. Hence, $\pi_1^{orb}(\orbQ)^{ab} \twoheadrightarrow \Z/2\Z.$  In light of this observation and \cite[Lemma 4.1(1)]{Hoffman_hidden}, we find that $\pi_1^{orb}(\orbQ)^{ab}\cong \Z/2\Z$  and $\orbQ$ has a 2-fold cover by an orbifold $\orbQ_3$ with a $S^2(3,3,3)$ cusp. $\orbQ_3$ is covered by $M$ since $\pi_1(\orbQ_3)$ and $\pi_1(\orbQ)$ have the same set of parabolic elements. By \cite[Lemma 5.5]{Hoffman_hidden}, the degree of the cover from $M$ to $\orbQ_3$ is $12n$ $(n\geq1)$. Thus, 
the covering degree of $M \to \orbQ$ is $24n$ ($n\geq 1$).
\end{proof}

 We now state a stronger version of \cite[Theorem 1.2]{Hoffman_hidden}, which follows from combining part (1) of that theorem with \Cref{thm:s236case} of this paper.
 
 \begin{cor}\label{cor:except}
Let $M$ be a manifold covered by a small non-arithmetic hyperbolic knot complement $\Sth \setminus K$ that admits two exceptional surgeries.
Then $M$ does not cover a rigid cusped orbifold. 
\end{cor}


\subsection{Rotationally rhombic peripheral groups}

Given a rigid cusped orbifold group, we can look at subgroups of the translation group. We say a translation group is \emph{rotationally rhombic} if it can be generated by two translations of the same length that are equivalent under a rotation of order $3$, $4$ or $6$. 
We stress that this rotation is only defined in terms of its action on the translation subgroup and it does not need to extend to a symmetry of to the associated 3-orbifold. 
Of course, each rigid cusped orbifold group has a maximal abelian subgroup that is rotationally rhombic and that abelian subgroup in turn has a subgroup  of index $(n_1^2+3n_2^2)$ or $(n_1^2+n_2^2)$ which is determined by the square of the norm of algebraic integer in $\Z[\frac{1+\sqrt{-3}}{2}]$ or $\Z[i]$. These abelian groups are rotationally rhombic and index $6(n_1^2+3n_2^2)$, $4(n_1^2+n_2^2)$, or $3(n_1^2+3n_2^2)$ in the rigid cusped orbifold group.  More generally, we say a cusped hyperbolic orbifold $\orbQ$ has a \emph{rotationally rhombic cusp} if the peripheral subgroup associated to this cusp is rotationally rhombic. 


Although we might expect to see an orbifold with a rotationally rhombic cusp in the commensurability class of a knot complement which covers a rigid cusped orbifold, in fact this is impossible in the $S^2(2,3,6)$ and $S^2(3,3,3)$ cases by the following lemma. The extra hypothesis in the $S^2(2,4,4)$ case does not prevent us from later utilizing this result in the proof of \Cref{thm:no_s244cusps}.
Thus, in the end, this case is also irrelevant. 
\begin{lem}\label{prop:no_rhombus} 
Let $\Sth \setminus K$ be a knot complment that covers a rigid cusped orbifold $\orbQ$.
\begin{enumerate}
\item If $\orbQ$ has $4$-torsion on its cusp, then any torus-cusped orbifold $\orbQ_T$ covered by $\Sth \setminus K$ such that $\Sigma(\orbQ_T)$ is two embedded circles  does not have a rotationally rhombic cusp.  
\item If $\orbQ$ has $3$-torsion on its cusp, then any torus-cusped orbifold $\orbQ_T$ covered by $\Sth \setminus K$ does not have a rotationally rhombic cusp.  
\end{enumerate}
\end{lem}

A key component of this result follows from  \cite[Proposition 5.8]{BBoCWaGT}. However, a number of the statements in that paper sometimes exclude the case that $\Sth\setminus K$ covers a rigid cusped orbifold, 
which seems to be out of convenience, not necessity. We include a more or less self-contained proof below because it still remains relatively brief. We also use the notation that $\langle \langle a_1, ..., a_k \rangle \rangle_{G}$ is the normal closure of the set $\{ a_j \}$ in $G$. 

\begin{proof}
The claims here can be observed directly for the figure 8 knot complement.

For all other knot complements, we appeal to \cite{ReidFig8} so that we can assume that $\Sth \setminus K$ is non-arithmetic. By Margulis \cite{Margulis1991}, $Comm^+(\pi_1(\Sth\setminus K))$ is discrete and for simplicity we assume $\orbQ \cong \Hth/Comm^+(\pi_1(\Sth\setminus K))$ (see also \cite[Proposition 9.1]{NeumReid}). 
As $Comm^+(\pi_1(\Sth\setminus K))$ is maximal,  if $\orbQ_T$ is covered by a knot complement, the cover $f\co \Sth\setminus K \rightarrow \orbQ$ factors through $f_2\co \orbQ_T \rightarrow \orbQ$.
Then, if we fix a subgroup $\Gamma_K $ such that $\Gamma_K \cong \pi_1(\Sth\setminus K)$, $\Gamma_K \subset \pi_1^{orb}(\orbQ_T) \subset Comm^+(\Gamma_K)$, and $Comm^+(\Gamma_K)$ is identified with its image under a discrete faithful representation into $\PSLTC$. After conjugation, we may assume that $\Hth/ Comm^+(\Gamma_K)$ has a cusp at $\infty$.  Let $\mu$ be a meridian of $\Gamma_K$ fixing $\infty$ and let $r$ be a rotation of order $3$ or $4$ on the cusp. 
As a group element, $\mu$, $r \mu r^{-1}$, $r^2 \mu r^{-2}$ are in $\Gamma_T=\pi_1^{orb}(\orbQ_T)$.  

{\bf Case 1:} if $r$ is order $4$, then $r^2 \mu r^{-2} = \mu^{-1}$. However, in this case, $\mu$ and $\mu'=r\mu r^{-1}$ correspond to meridians of knot groups 
$\Gamma_K$ and $r \Gamma_K r^{-1}$. Both of these knot groups are normal subgroups of $\Gamma_T$ by \cite[Lemma 4]{ReidFig8}
 (see also the proof of \cite[Lemma A.3]{baker2015poincare}) and both quotients are cyclic. In fact, we can say 
 $\Gamma_T/\Gamma_K \cong \Gamma_T/(r\Gamma_K r^{-1}) \cong \Z/n\Z$ for some fixed $n$.  Furthermore, 
 as $\Gamma_K$ and $r\Gamma_K r^{-1}$ are their own $\Gamma_T$-normal closures, 
 $\langle \langle \mu \rangle \rangle_{\Gamma_T}= \Gamma_K$ and $\langle \langle \mu' \rangle \rangle_{\Gamma_T}= \Gamma_{K'}$. 
 Since $\mu$ and $\mu'$ correspond to peripheral elements of  $\Gamma_T$, we have that 
 $\Gamma_T/\langle \langle \mu \rangle \rangle_{\Gamma_T} \cong \Gamma_T/\langle \langle \mu' \rangle \rangle_{\Gamma_T} \cong \Z/n\Z$. 
 Thus, by the Geometrization of cyclic orbifolds \cite{BP}, both of these fillings correspond to orbilens space fillings of $\orbQ_T$. 
 Furthermore, by assumption $\Sigma(\orbQ_T)$ consists of two embedded circles, and so $\orbQ_T \setminus \Sigma(\orbQ_T)$ 
 is a knot complement in $S^1 \times S^1 \times I$ that admits a non-trivial cosmetic surgery, which contradicts \cite[Lemma 5.3]{BBoCWaGT}. 

{\bf Case 2:} If $r$ is order $3$ then $\mu$, $\mu'=r \mu r^{-1}$, $\mu''=r^2 \mu r^{-2}$ are $3$ distinct elements of 
$\Gamma_T$. In fact, each one is a normal generator for one of the three isomorphic knot groups in $\Gamma_T$, 
$\Gamma_K$, $\Gamma_{K'}=r\Gamma_K r^{-1}$, and $\Gamma_{K''}=r^2\Gamma_K r^{-2}$. Aping the argument 
from above, we now have $\Gamma_T/\Gamma_K \cong \Gamma_T/\Gamma_{K'} \cong \Gamma_T/\Gamma_{K''} \cong \Z/n\Z$ 
for some fixed $n$. However when viewed as translations in the rigid cusp group (the translations for $S^2(2,3,6)$ and $S^2(3,3,3)$ cusp groups are described in  \Cref{sub:translation_structure}), one of $\mu$, $r \mu r^{-1}$, $r^2 \mu r^{-2}$ must be either be a homological sum of the other two. Depending on orientation, two of these having the homological quotient (here $\Z/n\Z$) would mean the third is either homologically trivial in $H_1(\orbQ_T) (=\Gamma_T^{ab})$ or has a quotient of $\Z/2n\Z$. 
Again as above $\langle \langle \mu \rangle \rangle_{\Gamma_T}= \Gamma_K$, 
$\langle \langle \mu' \rangle \rangle_{\Gamma_T}= \Gamma_{K'}$, and  $\langle \langle \mu'' \rangle \rangle_{\Gamma_T}= \Gamma_{K''}$, 
and so one of the quotients of $\Gamma_T$ should have an abelianization not equal to the other two, contradicting that all are $\Z/n\Z$. 
\end{proof} 


\begin{rem}
There are  manifolds and orbifolds with rotationally rhombic cusps that both cover a rigid cusped orbifold with $3$-torsion on the cusp and admit finite cyclic fillings. The figure 8 sister manifold `m003' is an example of such a manifold. Therefore, the property that $\langle \langle \mu \rangle \rangle_{\Gamma_T}= \Gamma_K$ is essential to the previous argument. For the figure 8 sister manifold the relevant quotients by the normal closures of peripheral elements are $\Z/5\Z$, $\Z/5\Z$ and $\Z/10\Z$, i.e. the figure 8 sister manifold has three lens space fillings one for each of those abelian quotients.
\end{rem}

For the discussion below, we will consider an orbifold $\orbQ$ that is covered by a knot complement and has an $S^2(2,4,4)$ cusp. In this case $|\orbQ|$ is simply connected and in fact an open ball by \cite[Proposition 2.3]{Hoffman_hidden}, equivalently \cite[Corollary 4.11]{BBoCWaGT}.  We will be interested in the cover $f: \Sth \setminus K \to \orbQ$ and especially properties of the cover intrinsic to the cusp. In that case, we can denote by $\bar{\orbQ} \cong \orbQ \setminus S^2(2,4,4) \times [1,\infty)$ and consider the restriction of $f: \Sth \setminus n(K) \to \bar{\orbQ}$. This is a cover of the knot exterior to the (closure of the interior) of the orbifold. 

In this case, $\bar{\orbQ}$ has underlying space a closed ball and $\Sigma({\orbQ})$ is a properly embedded trivalent graph. The vertices of this graph are fixed by non-cyclic finite subgroups of $SO(3,\R)$. We can label a vertex of this graph by the (maximal) finite subgroup of $SO(3,\R)$ fixing that point. Each edge of this graph is fixed by an elliptic element of $\Gamma$ and so we can label an edge using a finite cyclic group or more simply the order of  the (maximal) finite cyclic group that fixes that edge. For example, $\partial(\bar{\orbQ})$ is incident to an edge labeled $2$ and two edges labeled $4$.

\begin{proof}[Proof of \Cref{thm:no_s244cusps}]
Assume that $\orbQ$ has a $S^2(2,4,4)$ cusp and is covered by a knot complement $\Sth \setminus K$, then similar to the arguments used to prove \Cref{thm:s236case} we can say $\pi_1^{orb}(\orbQ) = P_4 \cdot \Gamma_K$, where $P_4 \cong \pi_1^{orb}(S^2(2,4,4)= \langle c, d | c^4,d^2,(cd)^4 \rangle$ and $\Gamma_K \cong \pi_1(\Sth \setminus K)$. 

Again just as above, we can assume that $\Gamma_K = \langle \mu_1, ... \mu_n | r_1, ... r_m \rangle$ and then build the following presentation 
$$\pi_1^{orb}(\orbQ) = \langle c, d, \mu_1, ...\mu_n |  r_1, ... r_m, c^4,d^2,(cd)^4, c\mu_j c^{-1}\gamma_{c,j}^{-1}, d\mu_j d^{-1}\gamma_{d,j}^{-1} \rangle $$ where (similar to the $S^2(2,3,6)$ case each $ \gamma_{c,j}^{-1}$, $\gamma_{d,j}^{-1}$ is a parabolic in $\pi_1^{orb}(\orbQ)$. 
The previous observation about the structure of these parabolics can be lightly adapted from that argument. Here, we still have that since $\Gamma_K$ has one orbit of parabolic fixed points (which is the same set of parabolic fixed points as $ \pi_1^{orb}(\orbQ)$),  and so $\gamma_{c,j}$ and $\gamma_{d,j}$ can be expressed as $w t_1^r t_2^s w^{-1}$, where $w \in \Gamma_K$ and $t_1 = c^2 d^{-1}$, $t_2 = cd^{-1}c$.

There is therefore a map $h: \pi_1^{orb}(\orbQ) \twoheadrightarrow \Z/2\Z$ given by the introducing the relations that $d,c^2$ are trivial. Under $h$, $\Gamma_K$ is trivial as are all elements corresponding to peripheral 2 torsion as well as all parabolic elements. Therefore any knot complement which covers $\orbQ$ also covers this two-fold cover $\orbQ_2$. However, $\orbQ_2$ has no 4-torsion on the cusp. In fact, it has a $S^2(2,2,2,2)$ cusp. Thus, the cover $\Sth \setminus K$ is regular by  \cite[Lemma 4]{ReidFig8} (see also \cite[Lemma A.3]{baker2015poincare}). 

This also shows we have the set of covering maps exhibited in \Cref{fig:s244quotients}: $f_1: \Sth \setminus K \to \orbQ_T$, $f_2: \orbQ_T \to \orbQ_2$, and $f_3:  \orbQ_2 \to \orbQ$. 

\begin{figure}

    \centerline{
\xymatrix{
\Sth \setminus K \ar[dr]^{n}   &\\
 & \orbQ_T \ar[d]_2\\
 & \orbQ_2 \ar[d]_2\\
 & \orbQ\\
}}
\caption{\label{fig:s244quotients} The restrictions on the structure of quotients of a manifold covered by a knot complement that covers an orbifold with a $S^2(2,4,4)$ cusp. Later, we show even this restricted case cannot occur. }
\end{figure}

\begin{claim}
$\orbQ_T$ is a knot in an orbilens space such that $\Sigma(\orbQ_T)$ is two circles.
\end{claim}

\begin{proof}[Proof of Claim:]
Note, $\orbQ_2$ is an orbifold with underlying space a ball (see \cite[Proposition 2.3]{Hoffman_hidden} and \cite[Corollary 4.11]{BBoCWaGT}). The latter reference also shows that $\orbQ_2$ has two-fold cover $\orbQ_T$ which is a knot complement in an orbilens space and that $\orbQ_T$ is cyclically covered by $\Sth \setminus K$. Thus the there are $0,2$ or $ 4$ points fixed by non-cyclic isotropy groups of $\orbQ_2$ which are all dihedral since $\orbQ_2$ is a dihedral quotient of $\Sth \setminus K$.  
We note that just as in the proof of \cite[Lemma 4.3]{Hoffman_hidden}, the 4-torsion on the cusp is either connected via a loop back to the cusp (no internal vertices), $D_4$ or $S_4$. However, $S_4$ does not have a
dihedral subgroup of order 12, so if $S_4$ is an isotropy group connected to peripheral 4-torsion, then this would contradict the property that $\orbQ_2$ is a dihedral quotient of $\Sth \setminus K$. This leaves two cases to check.

{\bf Case 1:} If either set of points fixed by peripheral 4-torsion terminates at a point fixed by a $D_4$ group (a dihedral group of order $8$), then both such sets terminate at points fixed by $D_4$. Moreover, both $D_4$ isotropy groups have $D_2$ (the dihedral group of order 4) subgroups which lift to $\pi_1^{orb}(\orbQ_2)$. The peripheral $2$-torsion must connect to a dihedral group of order $2(p)$ where $p \in \Z$ in order to lift to a dihedral group in $\pi_1^{orb}(\orbQ_2)$. By cusp killing this group must be a $(2,2,p=2k+1)$ isotropy group. We observe that two conjugates of this group lift faithfully to $\pi_1^{orb}(\orbQ_2)$. In $\pi_1^{orb}(\orbQ_2)$, there are also two dihedral groups of order 4. Thus,  $\Sigma(\orbQ_T)$ consists of one circle fixed by 2-torsion and one circle fixed by odd torsion.

{\bf Case 2:} Otherwise,  the 4-torsion on the cusp is part of a loop. Thus, appealing to \cite[Proposition 4.2]{Hoffman_hidden}, we have that $\Sigma(\orbQ_T)$ is made up of two components. 
\end{proof}

  However, $\orbQ_T$ also has a rotationally rhombic cusp contradicting \Cref{prop:no_rhombus}.
Thus, $S^3 \setminus K$ cannot cover an orbifold with a $S^2(2,4,4)$ cusp.
\end{proof}

%
There is one knot complement $S^3 \setminus 12n706$ which is known to have cusp field $\Q(i)$ (see \cite[\S 7.1]{GHH}). This knot complement has become known as Boyd's knot complement because Boyd showed it decomposes into regular ideal tetrahedra and regular ideal octahedra. In fact, that polyhedral decomposition is the canonical cell decomposition. A simple analysis of this decomposition shows it does not support 4-torsion on the cusp. Of course, \Cref{thm:no_s244cusps} gives an alternate proof of this fact as well.

\section{Non-orientable cusp types.}\label{sect:non_or}

As noted in the introduction, it is well-established that four of the possible seventeen Euclidean 2-orbifolds show up as cusp types of orbifolds covered by knot complements. While \Cref{thm:no_s244cusps} rules out $S^2(2,4,4)$ cusps, it also rules out $D^2(4;2)$ and $D^2(;2,4,4)$ cusps as well since any orbifold with such a cusp type would have an orientation double cover with a $S^2(2,4,4)$ cusp. As all parabolic elements lift to the orientation double cover, a knot complement covering an orbifold with a  $D^2(4;2)$, or $D^2(;2,4,4)$ cusp would also cover an orbifold with a $S^2(2,4,4)$ cusp, giving the desired contradiction.

\begin{figure}
\begin{tabular}{ccc}

\begin{tikzpicture}
\draw[thick] (0,0) -- (2,0) -- (2,2) -- (0,2) -- (0,0);
\draw[thick] (0,0) circle (1mm);
\draw[thick] (0,2) circle (1mm);
\draw[thick] (2,0) circle (1mm);
\draw[thick] (2,2) circle (1mm);
\draw[thick]  (0,0) -- (0,1) node {R};
\draw[thick]  (2,0) -- (2,1) node {R};
\draw[thick]  (0,0) -- (1,0) node {R};
\draw[thick]  (0,2) -- (1,2) node {R};
\end{tikzpicture}
& 

\begin{tikzpicture}
\draw[thick] (0,0) -- (2,0) -- (2,2) -- (0,2) -- (0,0);
\draw[thick] (0,2) circle (1mm);
\draw[thick] (2,2) circle (1mm);
\fill[black] (1,0) circle (1mm);

\draw[thick,->]  (1,0) -- (1.5,0);
\draw[thick,->]  (1,0) -- (.5,0);
\draw[thick]  (2,0) -- (2,1) node {R};
\draw[thick]  (0,0) -- (0,1) node {R};
\draw[thick]  (0,2) -- (1,2) node {R};
\end{tikzpicture}
&
\begin{tikzpicture}
\draw[thick] (0,0) -- (2,0) -- (2,2) -- (0,2) -- (0,0);
\fill[black] (,0) circle (1mm);
\fill[black] (1,2) circle (1mm);
\draw[thick,->]  (1,0) -- (1.5,0);
\draw[thick,->]  (1,0) -- (.5,0);
\draw[thick,->]  (1,0) -- (1.6,0);
\draw[thick,->]  (1,0) -- (.6,0);
\draw[thick,->]  (1,2) -- (1.5,2);
\draw[thick,->]  (1,2) -- (.5,2);
\draw[thick]  (0,0) -- (0,1) node {R};
\draw[thick]  (2,0) -- (2,1) node {R};
\end{tikzpicture}
\\
a)  $D^2(;2,2,2,2)$  & b) $D^2(2;2,2)$ & c) $D^2(2,2;R)$\\
\begin{tikzpicture}
\draw[thick] (0,0) -- (2,0) -- (2,2) -- (0,2) -- (0,0);
\draw[thick,->]  (0,0) -- (1,0);
\draw[thick,->]  (0,2) -- (1,2);
\draw[thick]  (0,0) -- (0,1) node {R};
\draw[thick]  (2,0) -- (2,1) node {R};

\end{tikzpicture}

 & 
 \begin{tikzpicture}
\draw[thick] (0,0) -- (2,0) -- (2,2) -- (0,2) -- (0,0);
\draw[thick,->]  (0,0) -- (1,0);
\draw[thick,->]  (2,2) -- (1,2);
\draw[thick]  (0,0) -- (0,1) node {R};
\draw[thick]  (2,0) -- (2,1) node {R};

\end{tikzpicture}
& 

%
 \\
 d)   $T_R$ & e)  $K_R$ & \\ 
\end{tabular}

\caption{\label{fig:relevant_cusps} Schematics for 2-orbifolds that are possible cusp types discussed in \Cref{sect:non_or}. Edges marked with an `R' are fixed by reflection. Unmarked unfilled circles indicate corner reflectors fixed by $\Z/2\Z \times \Z/2\Z$ and unmarked filled circles indicate cone points of order 2. Otherwise the marking determines the cone point or cone reflector. }

\end{figure}

\begin{figure}

\begin{tikzpicture}
\draw [thick] (0,0) -- (120:2) -- (2,0) -- cycle;
\draw[thick,->]  (0,0) -- (1,0);
\draw[thick,->]  (0,0) -- (120:1);
\fill[black] (0,0) circle (1mm);
\draw[thick]   (2,0) circle (1mm);
\draw[thick]   (120:2) circle (1mm);
\draw (120:2.3)  node {3};
\draw (2.2,0)  node {3};
\draw[thick]  (60:1)  node {R};
\draw (-.3,-.3)  node[anchor=south] {3};
\end{tikzpicture}

\caption{\label{fig:d3_3} A diagram for the orbifold $D^2(3;3)$. }
\end{figure}

The figure 8 knot complement and the dodecahedral knot complements both cover tetrahedral orbifolds. The non-orientable tetrahedral orbifolds covered by these knot complements have either $D^2(;3,3,3)$ or $D^2(;2,3,6)$ cusps. 
The figure 8 knot complement also covers the Gieseking manifold which has a Klein bottle cusp and the figure 8 knot complement modulo its full symmetry group has a $RP^2(2,2)$ cusp. 

There is also a quotient of the figure 8 knot complement with a $D^2(3;3)$ cusp (see Figure \ref{fig:d3_3}), which we now describe. The non-orientable minimum volume tetrahedral orbifold commensurable with the figure 8 has volume $\frac{v_0}{24}$ (see for example \cite{NR92b}) and a fundamental group with presentation:

$$\Gamma = \langle a, b, c, d  | a^2,b^2,c^2,d^2,(ab)^6,(bc)^3,(ca)^2,(ad)^2,(bd)^2,(cd)^3 \rangle $$

which has $\Z/2\Z \times \Z/2\Z$ for its abelianization. There are three index two subgroups $\{\Gamma_1 =ker(\Gamma^{ab}/\langle \langle b \rangle \rangle,\Gamma_2= ker(\Gamma^{ab}/\langle \langle ab \rangle \rangle,\Gamma_3 = ker(\Gamma^{ab}/\langle \langle a \rangle \rangle\}$ of $\Gamma$. We point out that each of these subgroups corresponds to a 1-cusped cover of $\Hth/\Gamma$. The argument is as follows. Observe that $P= \langle a,b,c \rangle$ is the peripheral subgroup of $\Gamma$. Now consider $P_i = \Gamma_i \cap P$.  We then have  that $P_1= \langle (ab),(bc),(ca)  \rangle$ which corresponds to a $S^2(2,3,6)$ cusp, $P_2= \langle aba,b,c \rangle$ which corresponds to a $D^2(;3,3,3)$ cusp, and $P_3= \langle a, bc \rangle$ which corresponds to a $D^2(3;3)$ as cusp. Note that in each case $[P:P_i]=2$, so each $\Hth/\Gamma_i$ is 1-cusped as claimed.
An observation what will be utilized repeatedly in this section is that each $\Gamma_i$ contains every parabolic element of $\Gamma$. (In fact, each $\Gamma_i$ contains $\PSLTOTh$ as a subgroup.) Thus, any knot complement that covers $\Hth/\Gamma$ also covers $\Hth/\Gamma_i$. In particular, each quotient $\Hth/\Gamma_i$ is covered by the figure 8 knot complement. (A nearly identical argument applies to the dodecahedral knot complements).


For the five remaining cusp types (see \Cref{fig:relevant_cusps}), each one has a reflection which would extend to a reflection
 in the commensurator. More specifically, this reflection would extend to a reflection symmetry of the knot as established by the following:
 
 \begin{lem}\label{lem:reflections_as_symms}
 Assume $\Sth \setminus K$ covers a hypberbolic orbifold $\orbQ$ such that the cusp cross section of $\orbQ$ admits a reflection. 
 If all elements of the peripheral subgroup of $\orbQ$ are order 2 or have infinite order, 
 then
 $\Gamma_K =\pi_1(\Sth \setminus K) \triangleleft \Gamma_\orbQ= \pi_1^{orb}(\orbQ) $.  
 \end{lem} 
 
 \begin{proof}
 The assumptions imply that the cusp cross-section of $\orbQ$ appears as one of the five orbifolds listed in \Cref{fig:relevant_cusps} a) - e).

 If the cusp cross section of $\orbQ$ is $D^2(;2,2,2,2)$, $D^2(2;2,2)$ or $D^2(2,2;R)$, in each case, $\orbQ$ has an orientation double cover with a $S^2(2,2,2,2)$ cusp. Call this cover $\orbQ_2$. 
 Since all parabolic elements of $\Gamma_\orbQ$ are in $\Gamma_{\orbQ_2} = \pi_1^{orb}(\orbQ_2)$, $\Gamma_K \subset \Gamma_{\orbQ_2}$. Furthermore, we can decompose $\Gamma_\orbQ$ into cosets $\Gamma_{\orbQ_2}, r \Gamma_{\orbQ_2}$, where $r$ is a (fixed) reflection of the cusp cross-section $C$ of $\orbQ$. 
 Separately, we can denote by $\Gamma_T$ the index 2 subgroup of $\Gamma_{\orbQ_2}$ corresponding to a torus cusped orbifold $\orbQ_T$ covered by $\Sth \setminus K$. 
We have that $r\Gamma_T r^{-1}=\Gamma_T$ as $\Gamma_T$ is the subgroup of  $\Gamma_{\orbQ_2}$ generated by parabolic elements. 
We then observe that $r\Gamma_Kr^{-1} \subset \Gamma_T$. Also, if we let $\mu \in \Gamma_K$ denote the image of  
a meridian of $K$ in $C$, then (the normal closure) $\langle \langle \mu \rangle \rangle_{\Gamma_T} = \Gamma_K$ and 
 $\langle\langle r \mu r^{-1} \rangle  \rangle_{\Gamma_T}= r \Gamma_K r^{-1}$. Since $r \mu r^{-1} = \mu^{-1}$,  $\Gamma_K=r \Gamma_K r^{-1}$ and  $\Gamma_K \triangleleft \Gamma_\orbQ$.

 If the cusp cross section of $\orbQ$ is $T_R$ or $K_R$, then $\orbQ$ has an orientation double cover with torus cusp. We can call this cover $\orb_T$.  Using a reflection defined analogously as above, we still have $r\Gamma_Kr^{-1} \subset \Gamma_T$, and so this case is just a ``streamlined'' version of the previous argument.  
 \end{proof}

\begin{rem}\label{rem:notation}
We point out $T_R$ and $K_R$ in \Cref{fig:relevant_cusps} appear in the list of parabolic orbifolds in \cite[Theorem 13.3.6]{Th_notes} as {\emph annulus} and {\emph M\"obius band}, respectively. Also, $D^2(2,2;R)$ seems to correspond to `$(2;2;)$'.  Assuming the last case contains a typo (the first `;' should be a `,'), we have that for each case the reflection in boundary is implied.

If orbifolds with boundary were to be included in the list in \cite[Theorem 13.3.6]{Th_notes}, it would expand to include the \emph{seventeen} quotients by wallpaper groups 
and all \emph{seven} frieze group quotients, which would include the M\"obius band and annulus, so departing from Thurston's notation by using $T_R$ and $K_R$ is meant to avoid confusion.
\end{rem}

%
%

We now give a proof of \Cref{thm:possible_non_or_cusps}. This theorem gives a complete classification of the cusp types of quotients covered by knot complements in the sense that either a quotient can be ruled out from this list or it is realized as the cusp of a quotient of hyperbolic knot complement.

\begin{proof}[Proof of \Cref{thm:possible_non_or_cusps}]
First, \Cref{thm:no_s244cusps} eliminates the orbifolds with 4-torsion on the cusp (i.e. $S^2(2,4,4)$, $D^2(4;2)$ and $D^2(;2,4,4)$). Then, combining  \Cref{lem:reflections_as_symms} with Boileau, Boyer, Cebanu and Walsh's observation that hyperbolic knot complements cannot
admit a reflection symmetry \cite[Proof of Lemma 2.1(1)]{BoiBoCWa2} eliminates $D^2(;2,2,2,2)$, $D^2(2;2,2)$, $D^2(2,2;R)$, $T_R$, and $K_R$ as possible cusps. Finally, the other nine cusp types appear as cusps of orbifolds covered by the figure 8 knot complement. \end{proof}

\subsection{Realization of cusps covered by link complements.}\label{sub:link_realization}
We conclude by observing that every cusp type is realized in the quotient of a hyperbolic link complement. In light of \Cref{thm:possible_non_or_cusps} and the discussion above we will restrict our attention to the eight types of cusps not covered by knot complements. Namely, $T_R$, $K_R$, $D^2(2,2;R)$, $D^2(2;2,2)$, $D^2(;2,2,2,2)$, $S^2(2,4,4)$, $D^2(4;2)$, and $D^2(;2,4,4)$.

 
\begin{figure}

    \centerline{
\xymatrix{
&&& & \Sth \setminus L \ar[dd]^{3} \ar[dr]^{2}   & &\\ 
&&& && \orbQ_{T_R} \ar[dd]_2 \ar[ddr]_2 &\\
&&&&  \orbQ_3   \ar[ddddl]^{8} \ar[ddddll]^{8} & &\\ 
&&& & & \orbQ_{K_R} & \orbQ_{D^2(2,2;R)}\\
 &&&  & & &&\\
&&&  & & &&\\
&&\orbQ_{D^2(2;2,2)}\ar[d]_2 \ar[dl]_2 &  \Hth/\PSLTOone \ar[d]_2 \ar[dl]_2& & &&\\
 &\orbQ_{D^2(4;2)}\ar[drr]_2& \orbQ_{D^2(;2,2,2,2)}\ar[dr]^2& \orbQ_{S^2(2,4,4)} \ar[d]_2&&&&\\
 &&& \orbQ_{D^2(;2,4,4)} &&&
}}
\caption{\label{fig:s244quotients_for_links} Quotients of the Borromean rings complement. We point out $\orbQ_{K_R}$ and $\orbQ_{D^2(2,2;R}$ are 2-cusped, while $\orbQ_{T_R}$ is 3-cusped. Compare this figure to Table \ref{table:covers}}
\end{figure}

\begin{table}
\begin{tabular}{|c|c|c|}
\hline
Orbifold & Group & covering degree \\ 
\hline
\hline
$\Sth \setminus L$ & $\Gamma_L = \langle t_1t_2,  r t_1t_2 r^{-1}, r^{-1}t_1t_2r \rangle$  &1 \\
\hline
$\orbQ_3$ & $\langle \Gamma_L, r \rangle$ & 3 \\
\hline
$\Hth/\PSLTOone $ & $\langle \Gamma_L, r, c^2,d \rangle$  & 24 \\
\hline
$\orbQ_{D^2(2;2,2)}$ & $\langle \Gamma_L, r, c^2,\alpha_3 \rangle$ & 24\\
\hline
$\orbQ_{D^2(;2,2,2,2)}$ & $\langle \Gamma_L, r, \alpha_2,\alpha_3 \rangle$ & 48\\
\hline
$\orbQ_{D^2(4;2)}$ &$\langle \Gamma_L, r, c,\alpha_3 \rangle$   & 48 \\
\hline
$\orbQ_{S^2(2,4,4)}$ & $\langle \Gamma_L, r, c,d \rangle$  & 48\\
\hline
$\orbQ_{D^2(;2,4,4)}$ & $\langle \Gamma_L, r, \alpha_1, \alpha_2, \alpha_3 \rangle$ & 96\\
\hline

\end{tabular}
%
\caption{\label{table:covers} More detailed information about the quotients of $\Sth \setminus L$ which are described in terms of presentations their subgroups as of $\Gamma_{D^2(;2,4,4)}$. The orbifolds are index by  name gives the cusp type  with each cusp type. }
\end{table}

Conveniently, for each cusp type $C$,  we can  an orbifold $\orbQ_C$ with at least one cusp of that type that is covered by the Borromean rings complement.  

First, we will discuss $T_R$, $K_R$, and $D^2(2,2;R)$. The Borromean rings complement decompose into two octahedra. There is a symmetry of these two octahedra (see for example \cite{Purcell_intro} and \cite[\S 3]{HMW19}). The quotient of the Borromean rings complement by this symmetry is an orbifold with three $T_R$ cusps. Also the Borromean rings admits a glide-reflection along the planar component which exchanges the two crossing circle cusps. This can be observed directly or using SnapPy \cite{snappy}.
 The quotient of this glide-reflection with the previous symmetry will be a two cusped orbifold. The quotient of the planar cusp will have the gluing pattern of $K_R$, while the quotient of the two crossing circle cusps will remain $T_R$. Finally, we can compose the reflection with a strong involution to obtain a quotient with three $D^2(2,2;R)$ cusps.

We find the other quotients by considering small volume orbifolds in the commensurability class of the Borromean rings . The (full) symmetry group of the tessellation of 
$\Hth$ by regular ideal octahedra is given by $\Gamma_{D^2(;2,4,4)}$:

\begin{equation}
\begin{aligned}
\Gamma_{D^2(;2,4,4)}= \langle  \alpha_1, \alpha_2, \alpha_3, \alpha_4  | & \alpha_1^2,\alpha_2^2,\alpha_3^2,\alpha_4^2,(\alpha_2 \alpha_1)^4, (\alpha_3 \alpha_2)^2,(\alpha_1 \alpha_3)^4, \\
& (\alpha_4 \alpha_1)^2,(\alpha_4 \alpha_2)^2,(\alpha_4 \alpha_3)^3 \rangle.
\end{aligned}
\end{equation}

Letting $c=\alpha_2 \alpha_1, d=\alpha_3 \alpha_2, r=\alpha_4 \alpha_3 $, we also have that $\Gamma_{D^2(;2,4,4)}=  \langle  \alpha_2,c, d, r\rangle$. Here the definitions of $c$ and $d$ is consistent with $\S$\ref{sub:translation_structure244}. Specifically, in the sense that the group of translations in 
$P_{D^2(;2,4,4)} $ is generated by $t_1= c^2 d = (\alpha_2\alpha_1)^2\alpha_3\alpha_2$ and $t_2 = cdc =  \alpha_2\alpha_1\alpha_3\alpha_1$

With these conventions, the fundamental group of the Borromean rings is given by $$\Gamma_L = \langle t_1t_2, rt_1t_2r^{-1}, r^{-1}t_1t_2r \rangle.$$ The remainder of this argument just relies on the fact that $\Gamma_L$ is a link group, which can be observed directly from its abelianization $\Z^3$ and the fact that it is generated (also normal generated) by three peripheral elements.

We can then directly construct orbifolds that cover $\Hth/\Gamma_{D^2(;2,4,4)}$ which are in turn covered by $\Sth\setminus L$ using the presentations above. We will factor these covers through $\orbQ_3 \cong \Hth/\langle \Gamma_L, r \rangle$ which has one torus cusp and is the quotient of the Borromean rings complement by an order 3 symmetry. This makes the accounting easier as determining the covering degree from $\orbQ_3$ to any orbifold quotient is completely determined by the cover on the cusp. 

What we call $\Gamma_{S^2(2,4,4)}$ in Figure \ref{fig:s244quotients_for_links} is more commonly  is more well known as  $\PGLTOone$. Thus we have 
$$\Gamma_{S^2(2,4,4)} = \langle c, d, r \rangle = \langle c, d, r, \Gamma_L \rangle \cong \PGLTOone.$$ 

To verify the details of this argument, we can use the identification $c \mapsto  \begin{pmatrix} \zeta_8^{-1} & 0 \\ 0 & \zeta_8 \end{pmatrix}$, $d \mapsto  \begin{pmatrix} i & i \\ 0 & -i \end{pmatrix} $, $r \mapsto  \begin{pmatrix} 1 & 1 \\ -1 & 0 \end{pmatrix}$ where $\zeta_8 = e^{i\pi/4}$. With these identifications, $t_1 \mapsto  \begin{pmatrix} 1 & -1 \\ 0 & 1 \end{pmatrix}$, 
 $t_2\mapsto  \begin{pmatrix} 1 & i \\ 0 & 1 \end{pmatrix}$, and $\alpha_2$ is the reflection that coincides with complex conjugation.

 \begin{figure}
 \begin{subfigure}[t]{.4\textwidth}
  \centering
 \resizebox{1.4in}{!}{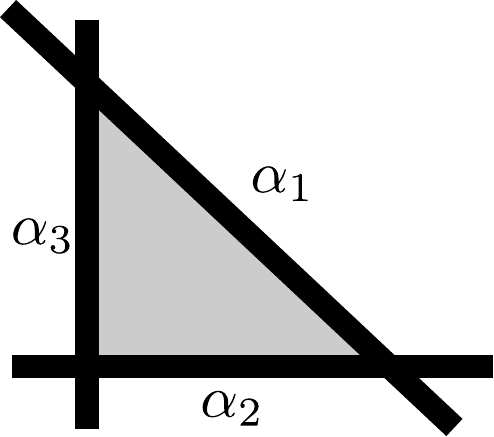}
   \caption{A fundamental domain for $P_{D^2(;2,4,4)}$ with the sides labeled by reflections $\alpha_1,\alpha_2,\alpha_3$}
  \label{fig:d224}
   \end{subfigure}
 \begin{subfigure}[t]{.4\textwidth}
 \resizebox{1.4in}{!}{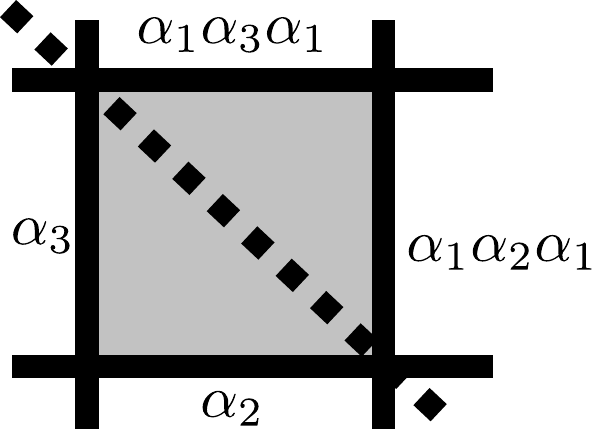}
   \caption{A fundamental domain for $P_{D^2(;2,2,2,2)}$}

  \centering
   \end{subfigure}
   \begin{subfigure}[t]{.4\textwidth}
  \centering
 \resizebox{1.4in}{!}{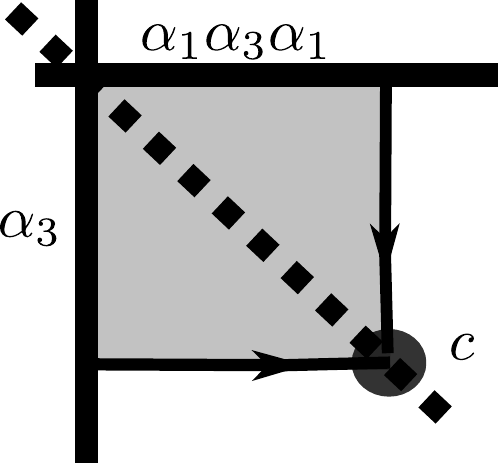}
   \caption{A fundamental domain for $P_{D^2(4;2)}$}
  \label{fig:d4r2}
    \end{subfigure}
   \begin{subfigure}[t]{.4\textwidth}
  \centering
 \resizebox{1.4in}{!}{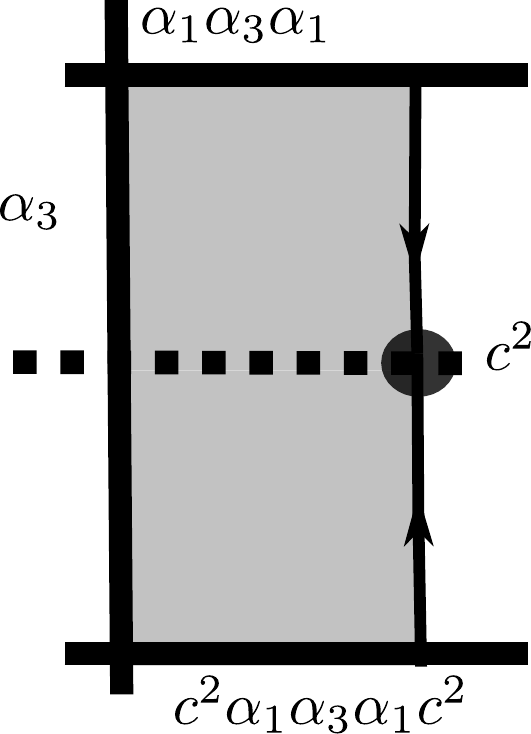}
   \caption{A fundamental domain for $P_{D^2(2;2,2)}$}
  \label{fig:d2r22}
   \end{subfigure}
 
 \caption{\label{fig:OtherRealizations} Realizations of covers of $D^2(;2,4,4)$ by $D^2(4;2)$, $D^2(;2,2,2,2)$ and $D^2(2;2,2)$}
 \end{figure}

We now consider the abelianization $f: \Gamma_{D^2(;2,4,4)} \rightarrow \Z/2\Z \times \Z/2\Z \times \Z/2\Z $. 
Here the peripheral subgroup is given by $P_{D^2(;2,4,4)} = \langle\alpha_1, \alpha_2, \alpha_3 \rangle$. We point out that image of $P_{D^2(;2,4,4)}$ under the abelianization is also $\Z/2\Z \times \Z/2\Z \times \Z/2\Z$. Thus, $P_{D^2(;2,4,4)}$ surjects $\Gamma_{D^2(;2,4,4)}^{ab}$. Also, for any abelian quotient both $r$ and $t_1 t_2$ are trivial. These observations show that any abelian quotient of  $\Gamma_{D^2(;2,4,4)}$ contains both $\langle r, \Gamma_L \rangle$ and $\Gamma_L$. 

We now have that $$\Gamma_{D^2(4;2)} = f^{-1}(\langle f(\alpha_2\alpha_1), f(\alpha_3)\rangle) = \langle \Gamma_L, r, \alpha_3, c \rangle,$$ $$\Gamma_{D^2(;2,2,2,2)} =f^{-1}(\langle f(\alpha_2), f(\alpha_3)\rangle) = \langle \Gamma_L, r, \alpha_2, \alpha_3 \rangle,$$ and 
$$\Gamma_{D^2(2;2,2)} =f^{-1}(\langle f(\alpha_3)\rangle) = \langle \Gamma_L, r, c^2, \alpha_3 \rangle.$$  These covers are pictured in Figure \ref{fig:OtherRealizations}.


%
%


We summarize the constructions in this section with the following proposition.

\begin{prop}\label{prop:link_quotients}
For every Euclidean 2-orbifold, there is a hyperbolic link complement which covers an orbifold with at least one cusp of that type. 
\end{prop}

\bibliographystyle{plain}
\bibliography{hidden_bib}

\end{document}